\documentclass{amsart}

\usepackage{ytableau}

\usepackage{tikz}
\usepackage{amssymb}
\usepackage{mathtools}

\newcommand{\al}{\alpha}

\newcommand{\la}{\lambda}
\newcommand{\La}{\Lambda}
\newcommand{\we}{\wedge}

\DeclareMathOperator{\st}{Std_{\la}(\mu)}
\DeclareMathOperator{\sta}{Std_{\al}(\mu)}
\DeclareMathOperator{\ST}{Std_{\la^{+}}(\mu^{+})}
\DeclareMathOperator{\STa}{Std_{\al^{+}}(\mu^{+})}
\DeclareMathOperator{\Hom}{Hom}

\newtheorem{theorem}{Theorem}[section]
\newtheorem{lemma}[theorem]{Lemma}
\newtheorem{proposition}[theorem]{Proposition}
\newtheorem{cor}[theorem]{Corollary}
\theoremstyle{definition}
\newtheorem{definition}[theorem]{Definition}
\newtheorem{example}[theorem]{Example}

\newtheorem{remark}[theorem]{Remark}
\newtheorem{remarks}[theorem]{Remarks}

\numberwithin{equation}{section}

\begin{document}

\title{Relating homomorphism spaces between Specht modules of different degrees}

\address{Department of Mathematics, University of Athens}
\curraddr{}
\email{mmaliak@math.uoa.gr}
\thanks{}

\author{Dimitra-Dionysia Stergiopoulou}
\address{Department of Mathematics, University of Athens}
\curraddr{}
\email{dstergiop@math.uoa.gr}
\thanks{}


\date{}

\dedicatory{}

\begin{abstract}Let $K$ be an infinite field of characteristic $p>0$ and let $\lambda, \mu$ be partitions of $n$, where $\lambda=(\lambda_1,...,\lambda_n)$ and $\mu=(\mu_1,..,\mu_n)$. By $S^{\lambda}$ we denote the Specht module corresponding to $\lambda$ for the group algebra $K\mathfrak{S}_n$ of the symmetric group $\mathfrak{S}_n$. D. Hemmer has raised the question of relating the homomorphism spaces $\Hom_{\mathfrak{S}_n}(S^{\mu}, S^{\lambda})$ and $\Hom_{\mathfrak{S}_{n'}}(S^{\mu^+}, S^{\lambda^+})$, where $n'=n+kp^d$, $\lambda^+ =\lambda+(kp^{d})$, $\mu^+=\mu+(kp^{d})$, and $d, k$ are positive integers. We show that these are isomorphic if $p$ is odd,  $p^d >\min\{\la_2, \mu_1-\lambda_1\}$ and $\mu_2 \le \lambda_1$.
\end{abstract}
\maketitle
\bibliographystyle{amsplain}

\section{Introduction}Let $K$ be an infinite field of characteristic $p>0$. For a partition $\la$ of $n$, let $S^{\la}$ be the Specht module for the group algebra $K\mathfrak{S}_n$ of the symmetric group $\mathfrak{S}_n$ defined in \cite{Gr}, Section 6.3. These modules play an important role in the representation theory of $\mathfrak{S}_n$. For example the determination of their composition factors and the corresponding multiplicities is a central open problem in the area. Relatively few results are known concerning homomorphism spaces $\Hom_{\mathfrak{S}_n}(S^{\mu}, S^{\la})$ between Specht modules. One of the general results on their dimensions are the row and column removal theorems of Fayers and Lyle \cite{FL}. See also the treatment by Kulkarni \cite{Ku} on such results for Weyl modules.

If $\la=(\la_1,...,\la_n)$ is a partition and $m$ a nonnegative integer, we denote by $\la+(m)$ the partition $(\la_1+m, \la_2,...,\la_n)$. In \cite{He}, Problem 5.4, Hemmer raised the question of finding general theorems that relate the spaces $\Hom_{\mathfrak{S}_n}(S^{\mu}, S^{\la})$ and  $\Hom_{\mathfrak{S}_{n'}}(S^{\mu^+}, S^{\la^+})$, where $n'=n+kp^{d}$,  $\la^+ =\la+(kp^{d})$, $\mu^+=\mu+(kp^{d})$ and $d, k$ are positive integers. Motivation for this was the work of Henke \cite{Hen} and Henke and Koenig \cite{HK} where, among other results, equalities of related decomposition numbers for Schur algebras of degrees $n$ and $n+kp^d$ for the general linear group were obtained. Additional motivation when $p=2$ comes from the results of Dodge and Fayers {\cite{DF} on decomposable Specht modules.

The purpose of this paper is to prove the following result.

\begin{theorem}
	Let $K$ be an infinite field of characteristic $p>2$, let $\la=(\la_1,...\la_n)$, $\mu=(\mu_1,...,\mu_n)$  be partitions of $n$ and let $k,d$ be positive integers. If $p^d >\min\{\la_2, \mu_1-\la_1\}$ and $\mu_2\le \la_1$, then \[\Hom_{\mathfrak{S}_n}(S^{\mu}, S^{\la})\simeq \Hom_{\mathfrak{S}_{n'}}(S^{\mu^+}, S^{\la^+}),\] 
	where $n'=n+kp^{d}$,  $\la^+ =\la+(kp^{d})$ and $\mu^+=\mu+(kp^{d}).$
\end{theorem}
We cast the proof in the context of Weyl modules for the general linear group $GL_n(K)$, see Theorem 2.1, and use a classic result of Carter and Lusztig \cite{CL} to descend to Specht modules, see Remark 2.2. Our proof is elementary, as it is based on computations with tableaux, and yields an explicit isomorphism on the level of homomorphism spaces of Weyl modules.

In Section 2 we state the main result and in Section 3 we gather the necessary recollections for its proof  which is given in Section 4.

\section{Main result}
\subsection{Notation}

Throughout this paper, $K$ will be an infinite field of
characteristic $p>0$. Also we will fix positive integers $n$ and $r$. We will be working mostly with homogeneous polynomial representations of $GL_n(K)$ of degree $r$, or equivalently, with modules over the Schur algebra $S=S_K(n,r)$. A standard reference here is  \cite{Gr}.

Let $V=K^n$ be the natural $GL_n(K)$-module. The divided power algebra $DV=\sum_{i\geq 0}D_iV$ of $V$ is defined as the graded dual of the Hopf algebra $S(V^{*})$, where $V^{*}$ is the linear dual of $V$ and $S(V^{*})$ is the symmetric algebra of $V^{*}$, see \cite{ABW}, I.4. 

By $\wedge(n,r)$ we denote the set of sequences $\al=(\al_1, \dots, \al_n)$ of length $n$ of nonnegative integers that sum to $r$ and by $\wedge^+(n,r)$ we denote the subset of $\wedge(n,r)$ consisting of sequences $\lambda=(\lambda_1, \dots, \lambda_n)$ such that $\lambda_1 \ge \lambda_2 \dots \ge \lambda_n$. Elements of $\wedge^+(n,r)$ are referred to as partitions (of $r$ with at most $n$ parts). The transpose partition $\la^t =(\la_1^t,...,\la_n^t)$ of a partition $\la=(\la_1,...,\la_n)$ is defined by $\la_j^t = \# \{i: \la_i \ge j\}$, $j=1,...,\la_1$. Note that $\la^t \in \we^{+}(\la_1,r)$.

If $\al=(\al_1,\dots, \al_n) \in \wedge(n,r)$, we denote by $D(\al)$ or $D(\al_1,\dots,\al_n)$ the tensor product $D_{\al_1}V\otimes \dots \otimes D_{\al_n}V$. All tensor products in this paper are over $K$.

The exterior algebra of $V$ is denoted $\Lambda V=\sum_{i\geq 0}\Lambda^iV$. If $\al=(\al_1,\dots, \al_n) \in \wedge(n,r)$, we denote by $\Lambda(\al)$  the tensor product $\Lambda^{\al_1}V\otimes \dots \otimes \Lambda^{\al_n}V$.

For $\lambda \in \wedge^+(n,r)$, we denote by $\Delta(\lambda)$ the corresponding Weyl module for $S$. In \cite{ABW}, Definition II.1.4, the module $\Delta(\la)$ (denoted $K_{\la}F$ there), was defined as the image a map $d'_{\la} : D(\la) \to \Lambda(\la^{t})$. For example, if $\lambda =(r)$, then $\Delta(\lambda) =D_rV$, and if $\lambda =(1^r)$, then $\Delta(\lambda) =\La^{r}V$.

If $\la =(\la_1,...,\la_n)$ is a partition of $r$ and $m$ a nonnegative integer, we denote by $\la+(m)$ the partition $(\la_1+m,\la_2,...,\la_n)$ of $r+m$.

\subsection{Main result} The main result of this paper is the following.
\begin{theorem}Let $K$ be an infinite field of characteristic $p>0$ and let $\la, \mu$ be partitions in $\wedge^+(n,r)$ such that $\mu_2 \le \la_1$. Suppose $k,d$ are nonnegative integers such that $p^d >\min\{\la_2, \mu_1-\la_1\}$. Then \[\dim_K\Hom_S(\Delta(\la), \Delta(\mu))=\dim_K\Hom_{S'}(\Delta(\la^+), \Delta(\mu^+)),\]
	where $S=S_K(n,r), S^{'}=S_K(n, r+kp^d)$, $\la ^{+} = \la +(kp^d)$ and $\mu ^{+} = \mu +(kp^d).$
\end{theorem}

\begin{remark} Suppose $p>2$. Then from Theorem 3.7 of \cite{CL}, the vector spaces $\Hom_S(\Delta(\nu), \Delta(\nu'))$ and $\Hom_{\mathfrak{S}_n}(S^{\nu'}, S^{\nu})$ are isomorphic for all partitions $\nu, \nu'$ of $n$. Hence Theorem 1.1 follows from Theorem 2.1.
\end{remark}
\begin{example}We note that if either of the assumptions $p^d >\min\{\la_2, \mu_1-\la_1\}$ and $\mu_2 \le \la_1$ of Theorem 2.1 is relaxed, then the conclusion is not necessarily true, as the following examples show.\begin{enumerate}
		\item Let $p=3, \la=(8,3), \mu=(11)$ and $k=d=1$ so that $\la^+=(11,3)$ and $\mu^+=(14)$. Here we have $p^d =\min\{\la_2, \mu_1-\la_1\}$. The dimensions of the corresponding Hom spaces are \begin{align*}&\dim_K\Hom_S(\Delta(\la), \Delta(\mu))=1, \\ &\dim_K\Hom_{S'}(\Delta(\la^+), \Delta(\mu^+))=0.\end{align*}
		\item Let $p=3, \la=(1,1,1,1), \mu=(2,2)$ and $k=d=1$ so that $\la^+=(4,1,1,1)$ and $\mu^+=(5,2)$. Here we have $ \mu_2 > \la_1$. The dimensions of the corresponding Hom spaces are \begin{align*}&\dim_K\Hom_S(\Delta(\la), \Delta(\mu))=1, \\ &\dim_K\Hom_{S'}(\Delta(\la^+), \Delta(\mu^+))=0.\end{align*}
	\end{enumerate}
The above dimensions follow, for example, from  Theorem 3.1 of \cite{MS1} since in each case a hook is involved.
\end{example}
\begin{example}
	In \cite{MS2}, Example 6.4, we considered examples of Hom spaces of dimension greater than 1. Fix a prime $p>0$ and let $a$ be an integer such that $a>(p^2+1)(p-1)$ and $a \equiv p-2 \mod p^2.$ Define the  partitions  \begin{align*}&\la(a)=(a, 2p-1, (p-1)^{p^2}), \\& \mu(a)=(a+p, (p^2+1)(p-1)),\end{align*}
	where $p-1$ appears $p^2$ times as a part of $\la(a)$. It was shown in \cite{MS2} that \[\dim_K\Hom_S(\Delta(\la(a)), \Delta(\mu(a)))>1.\] Now according to Theorem 2.1, all of the above spaces have the same dimension (as $a$ varies and $p$ is fixed).
	
	For example,  let $p=3$. The least $a$ satisfying the above conditions is $a=28$. Using the GAP4 program written by M. Fayers \cite{F}, we have \[\dim_K\Hom_S(\Delta(\la(28)), \Delta(\mu(28)))=2.\] Hence \[\dim_K\Hom_S(\Delta(\la(a)), \Delta(\mu(a)))=2\]
	for all $a \ge 28$ such that $a \equiv 1 \mod 9$.
\end{example}

\section{Recollections}
In this section we recall some results need for the proof of Theorem 2.1.

\subsection{Relations for Weyl modules.} We recall from \cite{ABW}, Theorem II.3.16, the following description of $\Delta (\lambda)$ in terms of generators and relations. \begin{theorem}[\cite{ABW}] Let $\lambda=(\lambda_1,\dots,\lambda_m) \in \wedge^+(n,r)$, where $\lambda_m >0$. There is an exact sequence of $S$-modules \[
\sum_{i=1}^{m-1}\sum_{t=1}^{\lambda_{i+1}}D(\lambda_1,\dots,\lambda_i+t,\lambda_{i+1}-t,\dots,\lambda_m) \xrightarrow{\square_{\la}} D(\lambda) \xrightarrow{d'_\lambda} \Delta(\lambda) \to 0,
\]
where the restriction of $ \square_{\la} $ to the summand $M(t)=D(\lambda_1,\dots,\lambda_i+t,\lambda_{i+1}-t,\dots,\lambda_m)$ is the composition
\[
M(t) \xrightarrow{1\otimes\cdots \otimes \Delta \otimes \cdots 1}D(\lambda_1,\dots,\lambda_i,t,\lambda_{i+1}-t,\dots,\lambda_m)\xrightarrow{1\otimes\cdots \otimes \eta \otimes \cdots 1} D(\lambda),
\]
where $\Delta:D(\lambda_i+t) \to D(\lambda_i,t)$ and $\eta:D(t,\lambda_{i+1}-t) \to D(\lambda_{i+1})$ are the indicated components of the comultiplication and multiplication respectively of the Hopf algebra $DV$ and $d'_\lambda$ is the map in \cite{ABW}, Def.II.13.\end{theorem} 

\begin{remark}
	Let $\{e_1,...,e_n\}$ be the natural basis of $V=K^n$. According to \cite{AB}, equation (5) of Section 2, each $M_i(t)$ is a cyclic $S$-module generated by \[e_1^{(\la_1)} \otimes \cdots \otimes e_i^{(\la_i+t)} \otimes e_{i+1}^{(\la_{i+1}-t)} \otimes \cdots \otimes e_n^{(\la_n)}.\] Hence in order to show that a map of $S$-modules $\phi : D(\la) \to \Delta (\mu)$ induces a map $\Delta(\la) \to \Delta(\mu)$, it suffices to show that \[\phi(x_{i,t})=0\] for all $i=1,..,m-1$ and $t=1,...,\la_{i+1}$, where \[x_{i,t}=e_1^{(\la_1)} \otimes \cdots \otimes e_i^{(\la_i)} \otimes e_i^{(t)} e_{i+1}^{(\la_{i+1}-t)} \otimes \cdots \otimes e_n^{(\la_n)}.\]
\end{remark}

\subsection{Standard basis of $\Delta(\mu)$} We will record here a fundamental fact from \cite{ABW}

Consider the order $e_1<e_2<...<e_n$ on the natural basis $\{e_1,...,e_n\}$ of $V=K^n$. We will denote each element $e_i$ by its subscript $i$. For a partition $\mu=(\mu_1,...,\mu_n) \in \we^+(n,r)$, a tableau of shape $\mu$ is a filling of the diagram of $\mu$ with entries from $\{1,...,n\}$. The set of tableaux of shape $\mu$ will be denoted by $\mathrm{Tab}(\mu)$. 

A tableau is called \textit{standard} if the entries are weakly increasing across the rows from left to right and strictly increasing in the columns from top to bottom. (The terminology used in \cite{ABW} is 'co-standard'). The set of standard tableaux of shape $\mu$ will be denoted by $\mathrm{Std}(\mu)$.

The \textit{weight} of a tableau $T$ is the tuple $\alpha=(\alpha_1,...,\alpha_n)$, where $\alpha_i$ is the number of appearances of the entry $i$ in $T$. The subset of $\mathrm{Std}(\mu)$ consisting of the (standard) tableaux of weight $\alpha$ will be denoted by $\mathrm{Std}_{\alpha}(\mu).$ 

For example, let $n=5$. The following tableau of shape $\mu=(6,4,2,0,0)$ 
\begin{center}
$T=$
\begin{ytableau}
\ 1&1&1&2&2&4\\
\ 2&2&3&4\\
\ 2&5
\end{ytableau}
\end{center}
is not standard because of the violation in the first column. It has weight $\alpha=(3,5,1,2,1)$.

 We will use 'exponential' notation for tableaux. Thus for the above example we write \[ T=\begin{matrix*}[l]
1^{(3)}2^{(2)}4 \\
2^{(2)}34\\
25 \end{matrix*}. \]

To each tableau $T$ of shape $\mu=(\mu_1, ...,\mu_n)$ we may associate an element $x_T$  in $D(\mu)=D(\mu_1,...,\mu_n)$ \[x_T=x_T(1) \otimes \cdots \otimes x_T(n) \in D(\mu_1,...,\mu_n),\] where $x_T(i)=1^{(a_{i1})}...n^{(a_{in})}$ and $a_{ij}$ is equal to the number of appearances of $j$ in the $i$-th row of $T$. For $T$ in the previous example we have $x_T=1^{(3)}2^{(2)}4\otimes2^{(2)}34 \otimes 25$.
According to \cite{ABW}, Theorem II.2.16, we have the following. \begin{theorem}[\cite{ABW}]The set $\{d'_{\mu}(x_T): T \in \mathrm{Std}(\mu)\}$ is a basis of the $K$-vector space $\Delta({\mu})$.\end{theorem}

If $T \in \mathrm{Tab}(\mu)$, we will denote the element $d'_{\mu}(x_T) \in \Delta(\mu)$ by $[T].$

\subsection{Weight subspaces of $\Delta(\mu)$} Let $\nu \in \we(n,r)$ and  $\mu=(\mu_1,...,\mu_n) \in \we^+(n,r)$. According to \cite{AB}, equation (11) of Section 2, a basis of the $K$-vector space  \[\Hom_S(D(\nu), \Delta(\mu))\] is in 1-1 correspondence with set $\mathrm{Std}_{\nu}(\mu)$ of standard tableaux of shape $\mu$ and weight $\nu$. For the computations to follow, we need to make this correspondence explicit. 

Let $\al = (\al_1,..., \al_n) \in \we(n,r)$ and $T \in \mathrm{Std}_{\al}(\mu)$. Let $a_{ij}$ be the number of appearances of $j$ in the $i$-th row  of $T$. Since $T$ is standard, we have $a_{ij}=0$ if $i>j$. Hence $\al_j = \sum_{i \le j}a_{ij}$ for each $j$. In particular we have $\al_1=a_{11}$. 

For each $j=1,...,n$ consider the indicated component \begin{align*}\Delta:& D(\al_j) \to D(a_{1j},a_{2j},...,a_{jj}), \\&x \mapsto \sum_{s}x_s(a_{1j}) \otimes x_s(a_{2j}) \otimes \cdots \otimes x_s(a_{jj})\end{align*} of the comultiplication map of the Hopf algebra $DV$.
\begin{definition}
	With the previous notation, let $T \in \mathrm{Std}_{\al}(\mu)$. Define the map $\phi_T:D(\al) \to \Delta(\mu),$ \begin{align*} &x_1 \otimes x_2 \otimes ...\otimes x_n \mapsto \\& \sum_{s_2,...,s_n}d'_{\mu}\left(x_1x_{2s_2}(a_{12})...x_{ni_n}(a_{1n}) \otimes x_{2i_2}(a_{22})...x_{ni_n}(a_{2n}) \otimes \cdots \otimes x_{ns_n}(a_{nn}) \right). \end{align*}
\end{definition}
\begin{example}
	Let $\al=(a,4,2)$ and $\mu=(a+2,4)$, where $a \ge 2$. Let $x=1^{(a-2)}\otimes 1^{(2)}2^{(2)} \otimes 3^{(2)} \in D(a-2,4,2)$ and $T \in \mathrm{Std}_{\al}(\mu)$, where \[T= \begin{matrix*}[l]
		1^{(a)}2^{(2)} \\
		2^{(2)}3^{(2)} \end{matrix*} .\]
	The image of $x_2=1^{(2)}2^{(2)}$ under the map $\Delta: D(4) \to D(2,2)$ is $1^{(2)} \otimes 2^{(2)}+12 \otimes 12 +2^{(2)} \otimes 1^{(2)}$. Then the above definition yields \[\phi_T(x)=\binom{a+2}{2}\begin{bmatrix*}[l]
		1^{(a+2)} \\
		2^{(2)}3^{(2)} \end{bmatrix*}+\binom{a+1}{1}\begin{bmatrix*}[l]
		1^{(a+1)} 2\\
		123^{(2)} \end{bmatrix*}+ \begin{bmatrix*}[l]
		1^{(a)} 2^{(2)}\\
		1^{2}3^{(2)} \end{bmatrix*},\]
	where the binomial coefficients come from multiplication in the divided power algebra $DV$.
\end{example}
 \begin{proposition}[\cite{AB}] Let $\al \in \we(n,r)$ and $\mu \in \we^+(n,r)$. A basis of the $K$-vector space  $\Hom_S(D(\al), \Delta(\mu))$ is the set \[\{\phi_T: T \in \mathrm{Std}_{\al}(\mu)\}.\]
\end{proposition}

\subsection{Binomial coefficients mod p}

We will need the following well known property of binomial coefficients. It follows immediately, for example, from Lucas' theorem (which is Lemma 22.4 of \cite{Jam}).
\begin{lemma}\label{binom}
Let $p$ be a prime and $a,b, k$ nonnegative integers. If $d$ is a positive integer such that $p^{d}>b$, then \[\binom{a+kp^{d}}{b} \equiv \binom{a}{b} \mod p.\]
\end{lemma} 

\section{Proof of Theorem 2.1}
\subsection{Bijection of certain tableaux}
Consider positive integers  $n,r,k,d$. If $\al=(\al_1,...,\al_n) \in \we(n,r)$, let $\al^+=(\al+kp^d,\al_2,...,\al_n) \in \we(n, r+kp^d)$. If $T \in \mathrm{Tab}_{\al}(\mu)$, let $T^+ \in \mathrm{Tab}_{\al^+}(\mu^+)$ be obtained from $T$ by inserting $kp^d$ 1's in the beginning of the top row.

Recall from Proposition 3.6 that we have the basis elements $\phi_T$ of $ \Hom_S(D(\al),\Delta(\mu))$, where $T \in \sta$ . Let us denote by $\phi^+_{X}$ the corresponding basis elements of $ \Hom_{S'}(D(\al^+),\Delta(\mu^+))$, where $X \in \STa.$
\begin{lemma}Let $\al \in \we(n,r)$ and $\mu \in \we^+(n,r)$ such that $\mu_2 \le \al_1$. Then the maps \begin{align*}&\sta \to \STa, T \mapsto T^+\\ &\STa \to \sta, T^+ \mapsto T\end{align*}
	are inverses of each other and bijections. Hence the map  \begin{align*}\Hom_S(D(\la), \Delta(\mu)) &\to \Hom_{S'}(D(\la^+), \Delta(\mu^+)),\\ \sum_{T \in \st}c_T\phi_T &\mapsto \sum_{T \in \st}c_T\phi^+_{T^+}\end{align*}
	is an isomorphism.
\end{lemma}
\begin{proof} From the definition of standard tableau, it is clear that if $T \in \sta$, then $T^+ \in \STa.$ It is also clear that the first map of the lemma is injective.
	
	If $X \in \STa$, then the first row of $X$ contains $\al_1+kp^d$ 1's. Let $T \in \mathrm{Tab}_{\al}(\mu)$ be the tableau obtained from $X$ by deleting $kp^d$ 1's. Since $\mu_2 \le \al_1$, we have that $T$ is standard, i.e. $T \in \sta$. It is clear that $T^{+}=S$. \end{proof}
 We will show that the isomorphism of the above lemma induces an isomorphism \[\Hom_S(\Delta(\la), \Delta(\mu)) \simeq \Hom_{S'}(\Delta(\la^+), \Delta(\mu^+))\]
if $p^d >\min\{\la_2, \mu_1-\la_1\}$. In order to accomplish this, we will need to do computations with particular tableaux in which 1's appear only in the first row and the number of occurrences is at least $\la_1$.
\subsection{A key corollary}
\begin{definition}
	Let $A$ and $A^+$ be the subsets of $\mathrm{Tab}(\mu)$ and $\mathrm{Tab}(\mu^+)$ respectively consisting of all tableaux of the form \[ \begin{matrix*}[l]
		1^{(\la_1+t)}2^{(a_{12})}3^{(a_{13})} \cdots n^{(a_{1n})} \\
		2^{(a_{22})}3^{(a_{23})} \cdots n^{(a_{2n})} \\
		3^{(a_{33})} \cdots n^{(a_{3n})} \\
		\cdots\\
		n^{(a_{nn})}  \end{matrix*}, \;\; \;\; \begin{matrix*}[l]
		1^{(\la_1+kp^d+t)}2^{(a_{12})}3^{(a_{13})} \cdots n^{(a_{1n})} \\
		2^{(a_{22})}3^{(a_{23})} \cdots n^{(a_{2n})} \\
		3^{(a_{33})} \cdots n^{(a_{3n})} \\
		\cdots\\
		n^{(a_{nn})}  \end{matrix*}, \]
respectively, where $0 \le t \le \la_2$ and $a_{ij}$ are nonnegative integers.
\end{definition}
\begin{remarks} (1) We note that $\st \subseteq A$ and $\ST \subseteq A^+$ (since $t=0$ is allowed).\\
(2) It is clear that the map $A \to A^+, U \mapsto U^+$, is a bijection. Hence a typical element in the linear span of the set $\{[S] \in \Delta(\mu^+): S \in A^+\}$ may be written in the form  $\sum_{U \in A}c_U[U^+]$, where $c_U \in K$.
\end{remarks}

 \begin{lemma} Let $U \in A$. Write \begin{align*}
 		[U]&=\sum_{T \in \mathrm{Std}(\mu)}c_T[T], \; \text{and} \\
 		[U^+]&=\sum_{T \in \mathrm{Std}(\mu)}c_{T^+}[T^+], 
 	\end{align*}
 in $\Delta(\mu)$ and $\Delta(\mu^+)$ respectively, where $c_T, c_{T^+} \in K$. Then $c_T=c_{T^+}$ for all $T \in \mathrm{Std}(\mu)$.
 \end{lemma}
\begin{proof} First some notation. If $\nu =(\nu_1,...,\nu_n)$ is a partition and $T \in \mathrm{Tab}(\nu)$ is a tableau, let $\overline{\mu}=(\mu_2,...,\mu_n)$  be the partition obtained from $\mu$ by deleting the first part and let $\overline{T}$ be the tableau obtained from $T$ by deleting the first row. Note that $\overline{\mu}=\overline{\mu^{+}}$ and $\overline{T}=\overline{T^{+}}$. 
	
	Let $U \in A$. By Theorem 3.3 there exist standard tableaux $Y \in \mathrm{Std}(\overline{\mu})$ and coefficients $c_{Y} \in K$ such that \begin{equation} [\overline{U}] =[\overline{U^+}]=\sum c_{Y}[Y].\end{equation} Since $U \in A$, the tableaux $Y$ contain no $1$. Also, the first row of each $Y$ has length $\mu_2 \le \la_1$. Hence by attaching to the top of each $Y$ the top row 
	\[1^{(\la_1+t)}2^{(a_{12})}3^{(a_{13})} \cdots n^{(a_{1n})}\]
	of $U$, we obtain standard tableaux $T$ in 
	$\mathrm{Std}(\mu)$ such that $\overline{T}=Y$. Likewise, by attaching to the top of each $Y$ the top row 
	\[1^{(\la_1+t+kp^d)}2^{(a_{12})}3^{(a_{13})} \cdots n^{(a_{1n})}\]
	of $U^+$, we obtain the standard tableaux $T^+$ in 
	$\mathrm{Std}(\mu^+)$ and we have  $\overline{T^+}=Y$. Hence from equations (4.1) we obtain (as in the last paragraph of the proof of Lemma II.2.15 of \cite{ABW}) \begin{align*}
		[U]&=\sum_{T \in \mathrm{Std}(\mu)}c_{\overline{T}}[T], \; \text{and} \\
		[U^+]&=\sum_{T \in \mathrm{Std}(\mu)}c_{\overline{T}}[T^+].
	\end{align*}
	Since the sets $\{[T] \in \Delta(\mu): T \in \mathrm{Std}(\mu)\}$ and $\{[T^+] \in \Delta(\mu^+): T \in \mathrm{Std}(\mu)\}$  are linearly independent, the result follows.
	\end{proof}
\begin{cor}\label{cor}
	Let $c_U \in K$, where $U \in A$, and consider elements $\sum_{U \in A}c_U[U] \in \Delta(\mu)$ and $\sum_{U \in A}c_U[U^+] \in  \Delta(\mu^+)$. Then
	\[\sum_{U \in A}c_U[U] =0 \Leftrightarrow \sum_{U \in A}c_U[U^+] =0. \]
\end{cor}
\begin{proof} According to the previous lemma and Theorem 3.3 we may write  \begin{align*}
		[U]&=\sum_{T \in \mathrm{Std}(\mu)}c_{U,T}[T], \; \text{and} \\
		[U^+]&=\sum_{T \in \mathrm{Std}(\mu)}c_{U,T}[T^+], 
	\end{align*}
	in $\Delta(\mu)$ and $\Delta(\mu^+)$ respectively, where $c_{U,T} \in K$. By substituting we have \[\sum_{U \in A}c_U[U] =\sum_{T \in \mathrm{Std}(\mu)}\left(\sum_{U \in A}c_Uc_{U,T}\right)[T] \]
	and \[\sum_{U \in A}c_U[U^+] =\sum_{T \in \mathrm{Std}(\mu)}\left(\sum_{U \in A}c_Uc_{U,T}\right)[T^+]. \] The result follows from the last two equations and the fact that the sets $\{[T] \in \Delta(\mu): T \in \mathrm{Std}(\mu)\}$ and $\{[T^+] \in \Delta(\mu^+): T \in \mathrm{Std}(\mu)\}$  are linearly independent.
	\end{proof}

\subsection{Proof ot Theorem 2.1}

First we record from \cite{MS2}, Lemma 4.2, the following lemma that is a particular case of the straightening law for Weyl modules with two parts.

\begin{lemma}Let $\mu=(\mu_1, \mu_2)$ be a partition consisting of two parts and consider an element \[[T]=\begin{bmatrix*}
		1^{(a_1)}   \cdots   n^{(a_n)} \\
		1^{(b_1)}  \cdots  n^{(b_n)}
	\end{bmatrix*} \in \Delta(\mu).\] Then we have the following. \begin{enumerate} \item If $a_1+b_1 > \mu_1$, then $[T]=0$. \item If $a_1+b_1 \le \mu_1$, then \begin{align*}
			[T]=(-1)^{b_1}\sum_{i_2,...,i_n} \tbinom{b_2+i_2}{b_2} \cdots \tbinom{b_n+i_n}{b_n}\begin{bmatrix*}[l]
				1^{(a_1+b_1)}  2^{(a_2-i_2)}  \cdots   n^{(a_n-i_n)} \\
				\noindent 2^{(b_2+i_2)} \cdots n^{(b_n+i_n)}
			\end{bmatrix*},\end{align*} where the sum ranges over all nonnegative integers $i_2, ..., i_n$ such that $i_2+\cdots+i_n=b_1$ and $i_s \le a_s$ for all $s=2,...,n$. \end{enumerate}
\end{lemma}

The point of the above lemma is that in case (2), the coefficients that appear do not depend on $a_1$.

\textit{Proof of Theorem 2.1.}

If $\phi \in \Hom_S(D(\la), \Delta(\mu))$, then by Proposition 3.6 there exist $c_T \in K$ such that  \begin{equation}\phi = \sum_{T \in \st}c_T\phi_T.\end{equation} Define $\phi^+ \in  \Hom_S(D(\la^+), \Delta(\mu^+))$ by \begin{equation}\phi^+ =\sum_{T \in \st}c_T\phi^+_{T^+}.\end{equation} Recall from the last part of Lemma 4.1 that the correspondence $\phi \mapsto \phi^+$ is a bijection. We will show that $\phi$ induces a map $\Delta(\la)\to \Delta(\mu)$ if and only if $\phi^+$ induces a map $\Delta(\la^+) \to \Delta(\mu^+).$ By Remark 3.2, this is equivalent to showing that 
\begin{equation}\phi(x_{i,t})=0 \Leftrightarrow \phi^+(x^+_{i,t})=0\end{equation}
for all $i=1,...,m-1$ and $t=1,...,\la_{i+1}$, where \begin{align*}
	&x_{i,t}=1^{(\la_1)} \otimes \cdots \otimes i^{(\la_i)} \otimes i^{(t)}(i+1)^{(\la_{i+1}-t)} \otimes \cdots \otimes n^{(\la_n)}, 
	\\&x^+_{i,t}=1^{(\la_1+kp^d)} \otimes \cdots \otimes i^{(\la_i)} \otimes i^{(t)}(i+1)^{(\la_{i+1}-t)} \otimes \cdots \otimes n^{(\la_n)}.
\end{align*}

Let $ T \in \st$. Then \[ T=\begin{matrix*}[l]
	1^{(\la_1)}2^{(a_{12})}3^{(a_{13})} \cdots n^{(a_{1n})} \\
	2^{(a_{22})}3^{(a_{23})} \cdots n^{(a_{2n})} \\
	3^{(a_{33})} \cdots n^{(a_{3n})} \\
\cdots\\
n^{(a_{nn})}  \end{matrix*}, \]
where $\sum_{i}a_{ij}=\la_j$ for $j=2,...,n$. 
Hence for the corresponding tableau $ T^+ \in \ST$ we have \[ T^+=\begin{matrix*}[l]
	1^{(\la_1+kp^d)}2^{(a_{12})}3^{(a_{13})} \cdots n^{(a_{1n})} \\
	2^{(a_{22})}3^{(a_{23})} \cdots n^{(a_{2n})} \\
	3^{(a_{33})} \cdots n^{(a_{3n})} \\
	\cdots\\
	n^{(a_{nn})}  \end{matrix*}. \]
\textit{Relations from rows 1,2.}

Consider $x_{1,t}$, where $1 \le t \le \la_2$. From the definition of $\phi_T$ we have  \begin{equation}\phi_T(x_{1,t})=\sum_{j_1+j_2=t}\binom{\la_1+j_1}{j_1}\begin{bmatrix*}[l]
	1^{(\la_1+j_1)}2^{(a_{12}-j_1)}3^{(a_{13})} \cdots n^{(a_{1n})} \\
	1^{(j_2)}2^{(a_{22-j_2})}3^{(a_{23})} \cdots n^{(a_{2n})} \\
	3^{(a_{33})} \cdots n^{(a_{3n})} \\
	\cdots\\
	n^{(a_{nn})}  \end{bmatrix*}.\end{equation} According to Lemma 4.6 (1), we may assume in equation (4.2) that $\la_1+t \le \mu_1$. 

Likewise we have \[\phi_{T^+}(x_{1,t}^+)=\sum_{j_1+j_2=t}\binom{\la_1+j_1+kp^d}{j_1}\begin{bmatrix*}[l]
	1^{(\la_1+j_1+kp^d)}2^{(a_{12}-j_1)}3^{(a_{13})} \cdots n^{(a_{1n})} \\
	1^{(j_2)}2^{(a_{22-j_2})}3^{(a_{23})} \cdots n^{(a_{2n})} \\
	3^{(a_{33})} \cdots n^{(a_{3n})} \\
	\cdots\\
	n^{(a_{nn})} \end{bmatrix*},\] where again by Lemma 4.6 (1) we assume that $\la_1+t \le \mu_1$. Applying Lemma 3.7 we obtain  \begin{equation}\phi_{T^+}(x_{i,t}^+)=
\sum_{j_1+j_2=t}\binom{\la_1+j_1}{j_1}\begin{bmatrix*}[l]
	1^{(\la_1+j_1+kp^d)}2^{(a_{12}-j_1)}3^{(a_{13})} \cdots n^{(a_{1n})} \\
	1^{(j_2)}2^{(a_{22-j_2})}3^{(a_{23})} \cdots n^{(a_{2n})} \\
	3^{(a_{33})} \cdots n^{(a_{3n})} \\
	\cdots\\
n^{(a_{nn})}
 \end{bmatrix*}.\end{equation} Now we apply Lemma 4.6 (2) to each term in the right hand side of (4.5) and (4.6). Substituting in (4.2) and (4.3) we obtain  

\begin{align*}
	&\phi(x_{1,t})=\sum_{T \in \st}a_T[T],\\&\phi^+(x^+_{1,t})=\sum_{T \in \st}a_T[T^+],
\end{align*}
for some $a_T \in K$. By Remark 4.3 (1), we may apply Corollary 4.5 to conclude that equivalence (4.4) holds (for $i=1$).\\
\textit{Relations from rows i,i+1, where i$>$1.}

This is similar to the previous case, but simpler. Consider $x_{i,t}$, where $i>1$ and $1 \le t \le \la_2$. Then \begin{align}\nonumber\phi_T(x_{i,t})&=\\ \sum_{j_1+\dots j_{i+1}=t}\binom{a_{1i}+j_1}{j_1}\cdots\binom{a_{ii}+j_i}{j_i}&\begin{bmatrix*}[l]
		1^{(\la_1)}\cdots i^{(a_{1i}+j_1)}(i+1)^{(a_{1i+1}-j_1)} \cdots n^{(a_{1n})} \\
		2^{(a_{22})}\cdots i^{(a_{1i}+j_2)}(i+1)^{(a_{2i+1}-j_2)} \cdots n^{(a_{2n})}  \\
			3^{(a_{33})}\cdots i^{(a_{3i}+j_3)}(i+1)^{(a_{3i+1}-j_3)} \cdots n^{(a_{3n})}\\
		\cdots\\
		n^{(a_{nn})}  \end{bmatrix*}\end{align}
and likewise, \begin{align}\nonumber\phi_{T^{+}}(x_{i,t}^+)&=\\ \sum_{j_1+\dots j_{i+1}=t}\binom{a_{1i}+j_1}{j_1}\cdots\binom{a_{ii}+j_i}{j_i}&\begin{bmatrix*}[l]
		1^{(\la_1+kp^d)}\cdots i^{(a_{1i}+j_1)}(i+1)^{(a_{1i+1}-j_1)} \cdots n^{(a_{1n})} \\
		2^{(a_{22})}\cdots i^{(a_{1i}+j_2)}(i+1)^{(a_{2i+1}-j_2)} \cdots n^{(a_{2n})}  \\
		3^{(a_{33})}\cdots i^{(a_{3i}+j_3)}(i+1)^{(a_{3i+1}-j_3)} \cdots n^{(a_{3n})}\\
		\cdots\\
		n^{(a_{nn})}  \end{bmatrix*}.\end{align}
As before, substituting (4.7), (4.8) we have \begin{align*}
	&\phi(x_{i,t})=\sum_{T \in \st}b_T[T],\\&\phi^+(x^+_{i,t})=\sum_{T \in \st}b_T[T^+],
\end{align*}
for some $b_T \in K$. By Remark 4.3 (1) and Corollary \ref{cor} we conclude that equivalence (4.4) holds (for $i>1$).

Having shown (4.4) for all $i,t$, we obtain linear maps \begin{align*}
	&\Psi^+:\Hom_S(\Delta(\la), \Delta(\mu)) \to \Hom_{S'}(\Delta(\la^+), \Delta(\mu^+)),\\&
	\Psi^-:\Hom_{S'}(\Delta(\la^+), \Delta(\mu^+)) \to \Hom_{S}(\Delta(\la), \Delta(\mu))
\end{align*}
defined as follows. If $\gamma \in \Hom_S(\Delta(\la), \Delta(\mu)) $ is induced by $\sum_{T \in \st}c_T\phi_T$, then $\Psi^+(\gamma)$ is the map induced by $\sum_{T \in \st}c_T\phi^+_{T^+}$, and likewise, if $\delta \in \Hom_{S'}(\Delta(\la^+), \Delta(\mu^+)) $ is induced by $\sum_{T \in \st}c_{T}\phi^{+}_{T^+}$, then $\Psi^-(\delta)$ is the map induced by $\sum_{T \in \st}c_T\phi_T$.

From Lemma 4.1 it follows that the two compositions $\Psi^+ \circ \Psi^-$ and  $\Psi^- \circ \Psi^+$ are the corresponding identity maps and hence $\Psi^+$ is an isomorphism.

\end{document}